\documentclass{amsart}
\usepackage{enumerate}
\newtheorem{Thm}{Theorem}[section]
\newtheorem{Prop}[Thm]{Proposition}
\newtheorem{Cor}[Thm]{Corollary}
\newtheorem{Lem}[Thm]{Lemma}

\theoremstyle{remark}
\newtheorem{Rem}[Thm]{Remark}
\theoremstyle{problem}

\numberwithin{equation}{section}
\begin{document}

\title[Two-sided Pompeiu problem and discrete groups]
{The two-sided Pompeiu problem for discrete groups}

\author[P. A. Linnell]{Peter A. Linnell}
\address{Math  \\Virginia Tech \\ Blacksburg \\ VA  24061--1026
\\ USA}
\email{plinnell@math.vt.edu}

\author[M. J. Puls]{Michael J. Puls}
\address{Department of Mathematics \\
John Jay College-CUNY \\
524 West 59th Street \\
New York, NY 10019 \\
USA}
\email{mpuls@jjay.cuny.edu}
\thanks{The second author would like to thank the Office for the Advancement of Research at John Jay College for research support for this project}

\begin{abstract}
We consider a two-sided Pompeiu type problem for a discrete group $G$. We give necessary and sufficient conditions for a finite subset $K$ of $G$ to have the $\mathcal{F}(G)$-Pompeiu property. Using group von Neumann algebra techniques, we give necessary and sufficient conditions for $G$ to be a $\ell^2(G)$-Pompeiu group. 
\end{abstract}


\date{September 24th, 2020}
\maketitle

\section{Introduction}\label{Introduction}
Let $\mathbb{C}$ be the complex numbers, $\mathbb{R}$ the real numbers, $\mathbb{Z}$ the integers and $\mathbb{N}$ the natural numbers. Let $2 \leq n \in \mathbb{N}$ and let $K$ be a compact subset of $\mathbb{R}^n$ with positive Lebesgue measure. The Pompeiu problem asks the following: Is $f = 0$ the only continuous function on $\mathbb{R}^n$ that satisfies 
\begin{equation} 
 \int_{\sigma(K)} f \, dx = 0  \label{eq:contpomp}
\end{equation}
for all rigid motions $\sigma$? If the answer to the question is yes, then $K$ is said to have the Pompeiu property. It is known that disks of positive radius do not have the Pompeiu property, see  \cite[Section 6]{Zalcman80} and the references therein for the details. The Pompeiu problem and some of its variations have been studied in various contexts, see  \cite{BerensteinZalcman80, CareyKaniuthMoran91, PeyerimhoffSamior10, RawatSitaram97, ScottSitaram88, Zalcman80} and the references therein for more information. 

In \cite{ScottSitaram88} the following version of the Pompeiu problem was studied: Let $G$ be a unimodular group with Haar measure $\mu$. Suppose $K$ is a relatively compact subset of $G$ with positive measure. Is $f = 0$ the only function in $L^1(G)$ that satisfies 
\begin{equation}
\int_{gKh} f(x) \, d\mu =0    \label{eq:doubletrans} 
\end{equation}
for all $g, h \in G$? This question was studied further in \cite{CareyKaniuthMoran91}. The purpose of this note is to investigate (\ref{eq:doubletrans}) in the discrete group setting. 

For the rest of this paper, $G$ will always be a finitely generated discrete group and $\mathcal{C}$ will denote a class of complex-valued functions on $G$ that contain the zero function. Let $f$ be a complex-valued function on $G$. We shall represent $f$ as a formal sum $\sum_{g \in G} a_g g$, where $a_g \in \mathbb{C}$ and $f(g) = a_g$. Define $\mathcal{F}(G)$ to be the set of all functions on $G$, and let $\ell^p(G)$ be the functions in $\mathcal{F}(G)$ that satisfy $\sum_{g \in G} \vert a_g \vert^p < \infty$. The group ring $\mathbb{C}G$ consists of those functions where $a_g =0$ for all but finitely many $g$. The group ring can also be thought of as the functions on $G$ with compact support. Let $K$ be a finite subset of $G$. In this paper we will consider the following discrete version of (\ref{eq:doubletrans}): When is $f = 0$ the only function in $\mathcal{C}$ that satisfies 
\begin{equation}
\sum_{x \in gKh} f(x) = 0  \label{eq:discretedoubletrans} 
\end{equation}
for all $g, h \in G$? 

The following related Pompeiu type problem for discrete groups was investigated in \cite{Puls13}: When is $f = 0$ the only function in $\mathcal{C}$ that satisfies 
\begin{equation}
\sum_{x \in gK} f(x) = 0 \label{eq:discretelefttrans}
\end{equation}
for all $g \in G$?

It is not difficult to see that $f = 0$ is the only function in $\mathcal{C}$ that satisfies (\ref{eq:discretedoubletrans}) if it is the only function in $\mathcal{C}$ satisfying (\ref{eq:discretelefttrans}). 

We shall say that a finite subset $K$ of $G$ is a {\em $\mathcal{C}$-Pompeiu set} if $f =0$ is the only function in $\mathcal{C}$ that satisfies (\ref{eq:discretedoubletrans}). A {\em $\mathcal{C}$-Pompeiu group} is a group for which every nonempty finite subset is a $\mathcal{C}$-Pompeiu set. 

The identity element of $G$ will be denoted by 1. If $S \subseteq G$, then we will write $\chi_S$ to indicate the characteristic function on $S, \chi_S(g) =1$ if $g \in S$ and $\chi_S(g) =0$ if $g \notin S$. If $S$ consists of one element $g$, then $\chi_g$ will be the usual point mass concentrated at $g$. 

One of our main results is:
\begin{Thm} \label{doubletranspompeiuset}
Let $G$ be a discrete group and suppose $K$ is a finite subset of $G$. Let $I$ be the ideal in $\mathbb{C}G$ that is generated by $\chi_K$. Then $K$ is a $\mathcal{F}(G)$-Pompeiu set if and only if $I = \mathbb{C}G$.
\end{Thm}

We will show that as a consequence of this result, algebraically closed groups and universal groups are examples of $\mathcal{F}(G)$-Pompeiu groups. This contrasts sharply with the one-sided translation Pompeiu type problem studied in \cite{Puls13}, since there are nonzero functions in $\ell^1(G)$ -in fact $\mathbb{C}G$- for the above groups that satisfy (\ref{eq:discretelefttrans}).

Using group von Neumann algebra techniques we will prove the following characterization, which is a generalization of \cite[Corollary 2.7]{CareyKaniuthMoran91}, for $\ell^2(G)$-Pompeiu groups.
\begin{Thm} \label{l2Pompeiugroup}
Let $G$ be a discrete group. Then $G$ is a $\ell^2(G)$-Pompeiu group if and only if $G$ does not contain a nontrivial finite normal subgroup.
\end{Thm}

It appears that the first investigation of a Pompeiu type problem in the discrete setting was the paper \cite{Zeilberger78}. For discrete groups the Pompeiu problem with respect to left translations was studied in \cite{Puls13}. Pompeiu type problems for finite subsets of the plane were examined in \cite{KissLaczKovichVincze18}. An interesting connection between the Fuglede conjecture and the Pompeiu problem for finite Abelian groups was established in \cite{KissMalSomlaiVizer20}.

In Section \ref{preliminaries} we give some preliminary material and preliminary results, including giving an equivalent condition to (\ref{eq:discretedoubletrans}) in terms of convolution equations. We prove Theorem \ref{doubletranspompeiuset} in Section \ref{theorem1} and we also give examples of groups that are $\mathcal{F}(G)$-Pompeiu groups. In Section \ref{Theoremtwo} we discuss group von Neumann algebras and prove Theorem \ref{l2Pompeiugroup}.
\section{Preliminaries}\label{preliminaries} 
In this section we give some necessary background and prove some preliminary results. Let $f = \sum_{g\in G} a_g g$ and let $h = \sum_{g\in G} b_g g$ be functions on $G$. The convolution of $f$ and $h$ is given by
\[ f \ast h = \sum_{g,x \in G} a_g b_x gx = \sum_{g \in G} \left( \sum_{x \in G} a_{gx^{-1}}b_x \right)g. \]
Sometimes we will write $f \ast h (g) = \sum_{x \in G} f(gx^{-1}) h(x)$ for when the function $f \ast h$ is evaluated at $g$. With respect to pointwise addition and convolution, $\mathbb{C}G$ is a ring. Also if $f \in \ell^1(G)$ and $h \in \ell^p(G)$, then $f \ast h \in \ell^p(G)$. However if both $f$ and $h$ are in $\ell^2(G)$ it might be the case $f \ast h$ is not in $\ell^2(G)$.

For $g \in G$, the left translation of $f$ by $g$ is given by $L_gf(x) = f(gx)$ and the right translation of $f$ by $g$ is denoted by $R_gf(x) = f(xg^{-1})$, where $x \in G$. Note that $L_gf = \chi_{g^{-1}} \ast f$ and $R_gf = f \ast \chi_g$. For a function $f, \tilde{f}$ will denote the function $\tilde{f}(x) = f(x^{-1})$, where $x \in G$ and $\bar{f}$ will indicate the complex conjugate of $f$. The following simple lemma gives a useful characterization of (\ref{eq:discretedoubletrans}) in terms of convolution equations.

\begin{Lem}\label{equivalencedoubletrans}
Let $K$ be a finite subset of $G$, and let $f$ be a complex-valued function on $G$. Then $f$ satisfies (\ref{eq:discretedoubletrans}) if and only if $\widetilde{{\chi_K}} \ast L_gf = 0$ for all $g \in G$.
\end{Lem}
\begin{proof}
Let $g, h \in G$, then 
\begin{eqnarray*}
\sum_{x \in gKh} f(x)  &  = &  \sum_{x \in Kh} L_g f(x) \\
                                   & = & \sum_{k \in K} L_g f(kh)  \\
                                   & = & \sum_{x \in G}\chi_K (x) L_g f(xh)  \\
                                   & = & \widetilde{\chi_K} \ast L_g f (h).
\end{eqnarray*}
Thus $\sum_{x \in gKh} f(x) =0$ for all $g,h \in G$ if and only if $\widetilde{\chi_K} \ast L_g f =0$ for all $g \in G$. 
\end{proof}
A similar calculation shows that (\ref{eq:discretedoubletrans}) is also equivalent to $R_g f \ast \widetilde{\chi_K} = 0$ for all $g \in G$. Observe that $\widetilde{\chi_K} \ast L_g f = \widetilde{\chi_K} \ast \chi_{g^{-1}} \ast f$ and $R_g f \ast \widetilde{\chi_K} = f \ast \chi_g \ast \widetilde{\chi_K}$.
\begin{Rem}
It was shown in \cite[Proposition 2.3]{Puls13} that the single convolution equation $\chi_K \ast \tilde{f} =0$ is equivalent to (\ref{eq:discretelefttrans}).
\end{Rem}

If $X$ is a set, then $\vert X \vert$ will indicate the cardinality of $X$. Let $R$ be a ring and recall that the center of $R$ is the set of all elements in $R$ that commute under multiplication with all elements of $R$. An {\em idempotent} in $R$ is an element $e$ that satisfies $e^2 =e$. A central idempotent for $R$ is an idempotent contained in the center of $R$. In \cite{Puls13} it was shown that if $G$ contains a nonidentity element of finite order, then there is a finite subset $K$ of $G$ and a nonzero function in $\mathbb{C}G$ that satisfies (\ref{eq:discretelefttrans}). We shall see that this is not the case for (\ref{eq:discretedoubletrans}). What is true though is
\begin{Prop}\label{finitenormal}
Suppose $K$ is a nontrivial finite normal subgroup of $G$, then $K$ is not a $\mathcal{C}$-Pompeiu set for any class of functions that contains $\mathbb{C}G$. 
\begin{proof}
Write $\chi_K = \sum_{k\in K} k$ for the characteristic function on $K$. Also $\widetilde{\chi_K} = \chi_K$ since $K$ is a subgroup of $G$. Now,
\begin{equation*}
\chi_K \ast \chi_K = \sum_{k \in K} k\vert K \vert = \vert K \vert \chi_K,
\end{equation*}
thus $\chi_K \ast (\chi_K - \vert K \vert ) =0$. For $g \in G, \chi_K \ast \chi_g = \chi_g \ast \chi_K$ if and only if $gK = Kg$. Consequently, $\chi_K$ is in the center of $\mathbb{C}G$ since $K$ is normal in $G$. Let $g \in G$, then
\begin{eqnarray*}  \chi_K \ast L_g (\chi_K - \vert K \vert)   &  =   & \chi_K \ast \chi_{g^{-1}} \ast (\chi_K - \vert K \vert)\\
      &   =  &  \chi_{g^{-1}} \ast \chi_K \ast (\chi_K - \vert K \vert)  \\
      &  =  & 0. 
 \end{eqnarray*}     
      Hence, $\widetilde{\chi_K} \ast L_g (\chi_K - \vert K \vert) =0$ for all $g \in G$. Lemma \ref{equivalencedoubletrans} yields that $K$ is not a $\mathcal{C}$-Pompeiu set for any $\mathcal{C}$ containing $\mathbb{C}G$. 
\end{proof}
\end{Prop}
\begin{Rem}
In the above proof, $\frac{\chi_k}{\vert K \vert}$ is a central idempotent in $\mathbb{C}G$. We shall see that central idempotents play a critical role in the proof of Theorem \ref{l2Pompeiugroup}.
\end{Rem}

\section{Theorem \ref{doubletranspompeiuset}}\label{theorem1}
In this section we prove Theorem \ref{doubletranspompeiuset} and give some examples of groups that are $\mathcal{F}(G)$-Pompeiu groups.

Let $f = \sum_{g \in G} a_gg \in \mathbb{C}G$ and $h = \sum_{g \in G} b_g g \in \mathcal{F}(G)$. Define a map $\langle \cdot, \cdot \rangle \colon \mathbb{C}G \times \mathcal{F}(G) \rightarrow \mathbb{C}$ by 
\begin{equation*}
\langle f, h \rangle = \sum_{g \in G} a_g \overline{b_g}. 
\end{equation*}
For a fixed $h \in \mathcal{F}(G), \langle \cdot, h \rangle$ is a linear functional on $\mathbb{C}G$. Now suppose $T$ is a linear functional on $\mathbb{C}G$. Define $h(g) = \overline{T(g)}$ for each $g \in G$. Thus each linear functional on $\mathbb{C}G$ defines an element of $\mathcal{F}(G)$. Hence, the vector space dual of $\mathbb{C}G$ can be identified with $\mathcal{F}(G)$. 
\subsection{Proof of Theorem \ref{doubletranspompeiuset}}
We now prove Theorem \ref{doubletranspompeiuset}. Let $K$ be a finite subset of $G$ and let $I$ be the ideal in $\mathbb{C}G$ generated by $\chi_K$. We begin by showing that if $K$ is a $\mathcal{F}(G)$-Pompeiu set, then $I = \mathbb{C}G$. Now assume that $K$ is a $\mathcal{F}(G)$-Pompeiu set and $I \neq \mathbb{C}G$. Because $I$ is a subspace of $\mathbb{C}G$, there is a nonzero $f \in \mathcal{F}(G)$ for which $\langle \alpha, f \rangle =0$ for all $\alpha \in I$. 

Fix $g \in G$ and let $h \in G$. Since $I$ is an ideal, $R_hL_g \chi_K \in I$, which means $\langle R_hL_g \chi_K, f \rangle =0$. Now
\begin{eqnarray*}
\langle R_h L_g \chi_K, f \rangle & = & \sum_{y \in G} R_h L_g \chi_K(y) \overline{f(y)}  \\
                                                    &  =  &  \sum_{y \in G} \chi_K (gyh^{-1}) \overline{f(y)} \\
                                                    &  = & \sum_{y \in G} \chi_K(yh^{-1}) \overline{f(g^{-1}y}) \\
                                                    & = & (\widetilde{\chi_K} \ast L_{g^{-1}} \overline{f})(h).                                                    
\end{eqnarray*}
Thus $(\widetilde{\chi_K} \ast L_{g^{-1}} \overline{f}) (h) = 0$ for all $h \in G$. Consequently, $\widetilde{\chi_K }\ast L_{g^{-1}} \overline{f} = 0$ for all $g \in G$. Thus $\overline{f}$  is a nonzero function that satisfies (\ref{eq:discretedoubletrans}), contradicting our assumption that $K$ is a $\mathcal{F}(G)$-Pompeiu set. Hence, $I = \mathbb{C}G$. 

Conversely, assume $I = \mathbb{C}G$. We will finish the proof of the theorem by showing that $f=0$ is the only function that satisfies (\ref{eq:discretedoubletrans}). Set $\widetilde{I} = \{ \widetilde{\alpha} \mid \alpha \in I\}$. Now $\widetilde{I}$ is generated by $\widetilde{\chi_K}$ and $\widetilde{I} = \mathbb{C}G$ since $I =\mathbb{C}G$. Assume that $f \in \mathcal{F}(G)$ satisfies (\ref{eq:discretedoubletrans}). Then by Lemma \ref{equivalencedoubletrans}, $\widetilde{\chi_K} \ast L_gf = R_gf \ast \widetilde{\chi_K} = 0$ for all $g \in G$. We now obtain $f \ast \widetilde{I} = 0 =\widetilde{I} \ast f$ since $\widetilde{\chi_K}\ast L_gf = \widetilde{\chi_K} \ast \chi_{g^{-1}} \ast f$ and $R_g f \ast \widetilde{\chi_k} = f \ast \chi_g \ast \widetilde{\chi_k}$. Now $\chi_1 \in \widetilde{I}$ and $0 = \chi_1 \ast f = f$, thus $K$ is a $\mathcal{F}(G)$-Pompeiu set and the theorem is proved.

\subsection{Examples} 
Before we give examples of $\mathcal{F}(G)$-Pompeiu groups we need to define the augmentation ideal of a group ring. Define a map from $\mathbb{C}G$ into $\mathbb{C}$ by 
\begin{equation*}
\varepsilon \left( \sum_{g \in G} a_g g\right) = \sum_{g \in G} a_g.
\end{equation*} 
The map $\varepsilon$ is a ring homomorphism onto $\mathbb{C}$. The augmentation ideal of $\mathbb{C}G$, which we will denote by $\omega(\mathbb{C}G)$, is the kernel of $\varepsilon$. For information about $\omega(\mathbb{C}G)$ see \cite[Chapter 3]{Passman85}. If $K$ is a nonempty finite subset of $G$, then $\chi_k \notin \omega(\mathbb{C}G)$ due to $\varepsilon (\chi_K) = \vert K \vert$. The main result of \cite{BonHartPassSmith76} showed that the only nontrivial ideal in $\mathbb{C}G$ for algebraically closed groups and universal groups is $\omega(\mathbb{C}G)$. Thus the ideal generated by $\chi_k$ in these groups must be all of $\mathbb{C}G$. Therefore, algebraically closed groups and universal groups are $\mathcal{F}(G)$-Pompeiu groups. 
\begin{Rem}
These groups have elements of finite order, which implies that there are nonzero functions in $\mathbb{C}G$ that satisfy (\ref{eq:discretelefttrans}). See \cite[Section 2]{Puls13} for the details.
\end{Rem}

\section{Theorem \ref{l2Pompeiugroup}}\label{Theoremtwo}
In this section we will prove Theorem \ref{l2Pompeiugroup}. We begin by discussing group von Neumann algebras. For a more detailed explanation of group von Neumann algebras see \cite[Section 8]{Linnell98}. Recall that $\ell^2(G)$ is the set of all formal sums $\sum_{g \in G} a_g g$ for which $\sum_{g \in G} \vert a_g \vert^2 < \infty$. Furthermore, $\ell^2(G)$ is a Hilbert space with Hilbert bases $\{ g \mid g \in G\}$. For $f = \sum_{g \in G} a_g g \in \ell^2(G)$ and $h = \sum_{g \in G} b_g g \in \ell^2(G)$, the inner product $\langle f, h \rangle$ is defined to be $\sum_{g \in G} a_g \overline{b_g}$. If $ f \in \mathbb{C}G$ and $h \in \ell^2(G)$, then $f \ast h \in \ell^2(G)$. In fact, multiplication on the left by $f$ is a continuous linear operator on $\ell^2(G)$. Thus we can consider $\mathbb{C}G$ to be a subring of $\mathcal{B}(\ell^2(G))$, the set of bounded operators on $\ell^2(G)$. Denote by $\mathcal{N}(G)$ the weak closure of $\mathbb{C}G$ in $\mathcal{B}(\ell^2(G))$. The space $\mathcal{N}(G)$ is known as the {\em group von Neumann algebra} of $G$. For $T \in \mathcal{B}(\ell^2(G))$ the following are standard facts.
\begin{enumerate}[(i)]
\item $T \in \mathcal{N}(G)$ if and only if there exists a net $(T_n)$ in $\mathbb{C}G$ such that \\$\lim_{n \rightarrow \infty} \langle T_nu, v \rangle \rightarrow \langle Tu, v \rangle$ for all $u, v \in \ell^2(G)$.
\item $ T \in \mathcal{N}(G)$ if and only if $(Tf) \ast \chi_g = T(f \ast \chi_g)$ for all $g \in G$. \label{rightmap}
\end{enumerate}
Another way of expressing (\ref{rightmap}) is that $T \in \mathcal{N}(G)$ if and only if $T$ is a right $\mathbb{C}G$-map. Using (\ref{rightmap}) we can see that if $T \in \mathcal{N}(G)$ and $T\chi_1 = 0$, then $T\chi_g = 0$ for all $g \in G$ and hence $Tf =0$ for all $f \in \mathbb{C}G$. It follows that $T = 0$ and so the map defined by $T \mapsto T\chi_1$ is injective. Therefore the map $T \mapsto T\chi_1$ allows us to identify $\mathcal{N}(G)$ with a subspace of $\ell^2(G)$. Thus algebraically we have 
\[ \mathbb{C}G \subseteq \mathcal{N}(G) \subseteq \ell^2(G). \]
It is not difficult to show that if $f \in \ell^2(G)$, then $f \in \mathcal{N}(G)$ if and only if $f \ast h \in \ell^2(G)$ for all $h \in \ell^2(G)$. For $ f = \sum_{g \in G} a_g g \in \ell^2(G)$, define $f^{\ast} = \sum_{g \in G} \overline{a_g} g^{-1} \in \ell^2(G)$. Then for $f \in \mathcal{N}(G)$ we have $\langle f \ast u, v \rangle = \langle u, f^{\ast} \ast v\rangle$ for all $u, v \in \ell^2(G)$; thus $f^{\ast}$ is the adjoint operator of $f$.

Two elements $x, y$ in $G$ are said to be {\em conjugate} in $G$ if there exists a $g \in G$ for which $g^{-1} x g =y$. Recall that the conjugation action of $G$ on itself is an equivalence relation.  Suppose $C$ is a finite conjugacy class of $G$. Let $c = \sum_{x \in C} x$, then $c \in \mathbb{C}G$. The group ring elements $c$ are known as {\em finite class sums}. For $x \in G$, denote by $C_x$ the class containing $x$. We will need the following 
\begin{Lem} \label{centerofN} 
Let $\mathcal{S}$ be the set of all finite class sums of $G$. Each element in the center of $\mathcal{N}(G), \mathcal{Z}(\mathcal{N}(G))$, is a formal sum of elements in $\mathcal{S}$. 
\end{Lem}
\begin{proof}
Let $f = \sum_{g \in G} a_g g \in \mathcal{N}(G)$. Then $f \in \mathcal{Z}(\mathcal{N}(G))$ if and only if $\chi_{g^{-1}} \ast f \ast \chi_g = f$ for all $g \in G$. Since
\[ \chi_{g^{-1}} \ast f \ast \chi_g = \sum_{x \in G} a_x (g^{-1}xg) = \sum_{y \in G} a_{gyg^{-1}} y, \]
we see immediately that $a_{gyg^{-1}} = a_y$ because $( \chi_{g^{-1}} \ast f \ast \chi_g)(y) = f(y)$. Thus $f$ is constant on $C_y$. If $f(y) \neq 0$ on $C_y$, then $C_y$ is finite due to $f \in \ell^2(G)$. The class sums have disjoints supports, thus if $f \in \mathcal{Z}(\mathcal{N}(G))$ then it is a formal sum of finite class sums in $\mathcal{S}$. 
\end{proof}
The {\em finite conjugate subgroup} of $G$ is defined by
\[ \Delta (G) = \{g \in G \mid g \hbox{ has a finite number of conjugates} \}. \]
The following immediate consequence of Lemma \ref{centerofN} will be crucial in our proof of Theorem \ref{NGPompeiugroup}.
\begin{Cor}\label{subsetcenter}
The center of $\mathcal{N}(G)$ is contained in the center of $\mathcal{N}(\Delta G).$
\end{Cor}
\begin{proof}
Every finite class sum contained in $\mathbb{C}G$ is contained in $\mathbb{C}(\Delta G)$.
\end{proof}

We will prove Theorem \ref{l2Pompeiugroup} by reducing to the $\mathcal{N}(G)$ case, which we now prove. 
\begin{Thm} \label{NGPompeiugroup}
If $G$ is a group with no nontrivial finite normal subgroups, then $G$ is a $\mathcal{N}(G)$-Pompeiu group.
\end{Thm}
\begin{proof}
It follows from \cite[Lemma 4.1.5(iii)]{Passman85} that $\Delta(G)$ is a torsion-free Abelian group since $G$ has no nontrivial finite normal subgroups. Let $K$ be a finite subset of $G$ and let $I$ be the weakly closed ideal in $\mathcal{N}(G)$ generated by $\widetilde{\chi_K}$. Now suppose there exists a nonzero $f \in \mathcal{N}(G)$ that satisfies $\widetilde{\chi_K} \ast L_gf = 0$ for all $g \in G$. So $f$ belongs to the annihilator ideal $I^{\perp}$ of $I$ in $\mathcal{N}(G)$. Thus $\mathcal{N}(G) = I \bigoplus I^{\perp}$, where $\bigoplus$ denotes the direct sum. Thus there exists a nonzero central idempotent $e$ in $\mathcal{N}(G)$ for which $e \ast \mathcal{N}(G) = I^{\perp}$. Because $\chi_1 \in \mathcal{N}(G), e \in I^{\perp}$ and it follows that $I \ast e =0$. By Corollary \ref{subsetcenter}, $e$ also belongs to the center of $\mathcal{N}(\Delta G)$. Let $T$ be a right transversal for $\Delta(G)$ in $G$. Write $\widetilde{\chi_K} = \sum_{t \in T} \chi_{Kt} \ast t$, where $\chi_{Kt} \in \mathbb{C}(\Delta G)$. Now 
\begin{eqnarray*}
   0  & = \widetilde{\chi_K} \ast e  &  = (\sum_{t \in T} \chi_{Kt} \ast t) \ast e \\
       &                                            & = \sum_{t \in T} (\chi_{Kt} \ast e ) \ast t.                                         
\end{eqnarray*}
Thus $\chi_{Kt} \ast e = 0$ for each $t \in T$. Because $\widetilde{\chi_K} \neq 0$ there exists a $t'$ in $T$ for which $\chi_{Kt'} \neq 0$. This contradicts the fact that $\chi_{Kt'} \ast \beta \neq 0$ for all $0\neq \beta \in \ell^2(\Delta G)$, which was proved in \cite{Cohen79}. Hence, there is no nonzero $f \in \mathcal{N}(G)$ that satisfies $\widetilde{\chi_K} \ast L_g f =0$ for all $g \in G$. Therefore, every finite subset of $G$ is a $\mathcal{N}(G)$-Pompeiu set and $G$ is a $\mathcal{N}(G)$-Pompeiu group. 
\end{proof}
\subsection{Proof of Theorem \ref{l2Pompeiugroup}} We now prove Theorem \ref{l2Pompeiugroup}. We start with a definition. A nonzero divisor in a ring $R$ is an element $s$ such that $sr \neq 0 \neq rs$ for all $r \in R\setminus 0$. 

Proposition \ref{finitenormal} says that if $G$ contains a nontrivial finite normal subgroup, then it cannot be  a $\ell^2(G)$-Pompeiu group.

Conversely, assume there exists a nonempty finite subset $K$ of $G$ and a nonzero $f \in \ell^2(G)$ that satisfies $\widetilde{\chi_K} \ast L_g f =0$ for all $g \in G$. By \cite[Lemma 7]{Linnell92} there exists a nonzero divisor $\theta \in \mathcal{N}(G)$ such that $f \ast \theta \in \mathcal{N}(G)$. Suppose $f \ast \theta = 0$. Then $\theta^{\ast} \ast f^{\ast} =0$. If $e \in \mathcal{B}(\ell^2(G))$ is the projection from $\ell^2(G)$ onto $\overline{f^{\ast} \ast \mathbb{C}G}$, then $e \in \mathcal{N}(G)$ by \cite[Lemma 5]{Linnell91}, $e \neq 0$ because $f^{\ast} \neq 0$ and $\theta^{\ast} \ast e = 0$. Hence $e \ast \theta = 0$, contradicting the fact $\theta$ is a nonzero divisor in $\mathcal{N}(G)$, so $f \ast \theta \neq 0$. It follows from Theorem \ref{NGPompeiugroup} that there exists a $g \in G$ such that $\widetilde{\chi_k} \ast L_g (f \ast \theta) \neq 0$ since $0 \neq f \ast \theta \in \mathcal{N}(G)$. But $\widetilde{\chi_K} \ast L_g(f \ast \theta) = (\widetilde{\chi_K} \ast g^{-1} \ast f) \ast \theta =0$ because we are assuming $\widetilde{\chi_K} \ast L_g f =0$ for all $g \in G$, a contradiction. Hence, there does not exist a nonzero $f \in \ell^2(G)$ and a nonempty finite subset $K$ of $G$ for which $\widetilde{\chi_K} \ast L_g f = 0 $ for all $g \in G$. Therefore it  follows from Lemma \ref{equivalencedoubletrans} that $G$ is a $\ell^2(G)$-Pompeiu group, as desired.
\bibliographystyle{plain}
\bibliography{twosidedpompeiudiscretegroup}

\end{document}